\pdfoutput=1
\documentclass[letterpaper,    
               english,        
               12pt,           
               twoside]        
               {amsart}        
\usepackage[english]{babel}                   
\usepackage[bookmarks,                        
            bookmarksopen=true,               
            bookmarksnumbered=true,           
            pdfauthor={Christof Puhle},       
            pdftitle={On generalized
                      quasi-Sasaki manifolds},
            pdfkeywords={Almost contact metric structures,
                         contact manifolds,
                         generalized quasi-Sasaki manifolds,
                         semi-cosymplectic manifolds,
                         nearly cosymplectic manifolds,
                         connections with torsion},
            colorlinks,                       
            linkcolor=black,                  
            citecolor=black,                  
            urlcolor=black,                   
            plainpages=false]                 
            {hyperref}                        
\usepackage[nobysame]
           {amsrefs}                          
\usepackage{fourier}                          
\usepackage{bm}                               
\usepackage{color}                            
\advance\oddsidemargin by -0.93cm
\advance\evensidemargin by -0.92cm
\advance\topmargin by -0.55cm
\theoremstyle{plain}
\newtheorem{cor}{Corollary}[section]
\newtheorem{lem}{Lemma}[section]
\newtheorem{thm}{Theorem}[section]            
\newtheorem{prop}{Proposition}[section]
\theoremstyle{definition}
\newtheorem{exa}{Example}[section]
\newtheorem{rmk}{Remark}[section]

\newcommand{\be}[1][*]{\begin{equation#1}}
\newcommand{\ee}[1][*]{\end{equation#1}}
\newcommand{\ba}{\be\begin{aligned}}
\newcommand{\ea}{\end{aligned}\ee}
\newcommand{\barr}[1]{\begin{array}{#1}}
\newcommand{\earr}{\end{array}}
\newcommand{\bite}{\begin{itemize}}
\newcommand{\eite}{\end{itemize}}
\newcommand{\btab}[1]{\renewcommand{\arraystretch}{1.2}\begin{center}\begin{tabular}{#1}}
\newcommand{\etab}{\end{tabular}\end{center}\renewcommand{\arraystretch}{1.0}}
\newcommand{\bnum}{\begin{enumerate}}
\newcommand{\enum}{\end{enumerate}}
\newcommand{\bcen}{\begin{center}}
\newcommand{\ecen}{\end{center}}
\def\side#1{\ifvmode\leavevmode\fi\vadjust{\vbox to0pt{\vss
\hbox to 0pt{\hskip\hsize\hskip1em                     
\vbox{\hsize2cm\small\raggedright\pretolerance10000       
\noindent\textcolor{red}{#1}\hfill}\hss}\vbox to8pt{\vfil}\vss}}}
\newcommand{\x}{\ensuremath{\times}}
\newcommand{\op}{\ensuremath{\oplus}}
\newcommand{\ox}{\ensuremath{\otimes}}

\newcommand{\lb}{\ensuremath{\left}}
\newcommand{\rb}{\ensuremath{\right}}

\newcommand{\lan}{\ensuremath{\left\langle}}
\newcommand{\ran}{\ensuremath{\right\rangle}}
\newcommand{\hook}{\ensuremath{\lrcorner\,}}
\newcommand{\ra}{\ensuremath{\rightarrow}}

\newcommand{\setsep}{\ensuremath{\,\big|\,}}
\newcommand{\R}{\ensuremath{\mathbb{R}}}

\newcommand{\Lie}{\ensuremath{\mathcal{L}}}

\newcommand{\W}{\ensuremath{\mathcal{W}}}

\newcommand{\A}{\ensuremath{\mathcal{A}}}

\newcommand{\vphi}{\ensuremath{\varphi}}

\newcommand{\diag}{\ensuremath{\mathrm{diag}}}
\newcommand{\Id}{\ensuremath{\mathrm{Id}}}
\newcommand{\tr}{\ensuremath{\mathrm{tr}}}
\newcommand{\pr}{\ensuremath{\mathrm{pr}}}
\newcommand{\Ric}{\ensuremath{\mathrm{Ric}}}

\newcommand{\SL}{\ensuremath{\mathrm{SL}}}

\newcommand{\un}{\ensuremath{\mathfrak{u}}}
\newcommand{\Un}{\ensuremath{\mathrm{U}}}

\newcommand{\Orth}{\ensuremath{\mathrm{O}}}
\newcommand{\so}{\ensuremath{\mathfrak{so}}}

\begin{document}
\thispagestyle{empty}
\date{\today}
\title{On generalized quasi-Sasaki manifolds}
\author{Christof Puhle}
\address{
Institut f\"ur Mathematik \newline\indent
Humboldt-Universit\"at zu Berlin\newline\indent
Unter den Linden 6\newline\indent
10099 Berlin, Germany}
\email{\noindent puhle@math.hu-berlin.de}
\urladdr{www.math.hu-berlin.de/~puhle}
\begin{abstract}
We study $5$-dimensional Riemannian manifolds that admit an almost contact metric structure. In particular, we generalize the class of quasi-Sasaki manifolds and characterize these structures by their intrinsic torsion. Among other things, we see that these manifolds admit a unique metric connection that is compatible with the underlying almost contact metric structure. Finally, we construct a family of examples that are not quasi-Sasaki.
\end{abstract}
\maketitle
\setcounter{tocdepth}{1}
\bcen
\begin{minipage}{0.7\linewidth}
    \begin{small}
      \tableofcontents
    \end{small}
\end{minipage}
\ecen
\pagestyle{headings}
%
%
%
%
\section{Introduction}\noindent
An almost contact metric manifold is an orientable Riemannian manifold $\lb(M^{2k+1},g\rb)$ of dimension $2k+1$ such that there exists a reduction of the structure group of orthonormal frames of the tangent bundle to the unitary group $\Un\left(k\right)$ (see \cite{Gra59}). Consequently, these manifolds represent the odd-dimensional analog of almost complex manifolds. As shown in \cites{SH61,SH62}, an almost contact metric structure on $\lb(M^{2k+1},g\rb)$ can be equivalently defined as a triple $\lb(\xi,\eta,\vphi\rb)$ consisting of a vector field $\xi$ of length one, its dual $1$-form $\eta$ and an endomorphism $\vphi$ of the tangent bundle satisfying certain relations (see \autoref{sec:2} for details). There are many special types of almost contact metric manifolds, of which the articles \cites{CG90,CM92,Puh12} may serve as a general reference. In order to characterize these types, it is customary to use differential equations involving the so-called fundamental $2$-form $\Phi$ and the Nijenhuis tensor $N$ of $\lb(M^{2k+1},g,\xi,\eta,\vphi\rb)$ (see \autoref{sec:2} for definitions). For example, Sasaki manifolds (see \cites{SH61,SH62}) are almost contact metric manifolds such that
\ba
d\eta &=2\cdot\Phi,& N&=0,
\ea
whereas cosymplectic manifolds satisfy
\ba
d\eta &=0,& d\Phi&=0,& N&=0.
\ea
Attempting to unify Sasakian and cosymplectic geometry, Blair introduced the class of quasi-Sasaki manifolds (see \cite{Bla67}). These are defined by
\ba
d\Phi &= 0,& N&=0.
\ea
In this case, the vector field $\xi$ is automatically a Killing vector field (see \cite{Bla67}).

Generalizing the class of quasi-Sasaki manifolds (see \autoref{sec:3}), the purpose of this article is to study almost contact metric $5$-manifolds that satisfy the following conditions:
\bite
\item[i)] The tensor fields $d\Phi$ and $N$ vanish on the distribution orthogonal to $\xi$.
\item[ii)] The vector field $\xi$ is a Killing vector field.
\eite
In particular, we are interested in the construction of examples that are not quasi-Sasaki. Almost contact metric $5$-manifolds that comply with properties i) and ii) are called generalized quasi-Sasaki manifolds.

In dimension $5$ (i.e.\ $k=2$), almost contact metric manifolds can be classified with respect to the algebraic type of the corresponding intrinsic torsion tensor $\Gamma$ (see \cite{Fri03}). There are $10$ irreducible $\Un(2)$-modules $\W_1,\ldots,\W_{10}$ in the decomposition of the space of possible intrinsic torsion tensors (see \cite{Puh12}):
\be
\Gamma\in\W_1\op\ldots\op\W_{10}.
\ee
Therefore, there exist $2^{10}=1024$ classes according to this approach. Obviously, most of them have never been studied. We review the special types mentioned above, in the light of this classification scheme, in \autoref{sec:3}: quasi-Sasaki manifolds correspond to the case $\Gamma\in\W_3\op\W_5$, cosymplectic manifolds are characterized by $\Gamma=0$, and generalized quasi-Sasaki manifolds correspond to the class $\W_3\op\W_4\op\W_5\op\W_7$ (see proposition \ref{prop:5}). Moreover, the considerations of \autoref{sec:3} (esp.\ theorem \ref{thm:1}) ensure that any generalized quasi-Sasaki manifold admits a unique metric connection $\nabla^c$ that is compatible with the underlying almost contact metric structure, i.e.\
\ba
\nabla^c\xi&=0, &
\nabla^c\eta&=0, &
\nabla^c\vphi&=0.
\ea
Viewed as a $\lb(3,0\rb)$-tensor field, the torsion tensor $T^c$ of the connection $\nabla^c$ is totally skew-symmetric if and only if $\Gamma\in\W_3\op\W_4\op\W_5$, and $T^c$ is traceless cyclic (i.e.\ contained in the third Cartan class; see \autoref{sec:3} for details) if and only if $\Gamma\in\W_7$ (see proposition \ref{prop:2}).

In \cite{Puh12}, the author presents explicit examples of almost contact metric $5$-mani\-folds whose intrinsic torsion tensors are of type
\be
\W_1,\W_2,\W_3,\W_5,\W_6,\W_8,\W_9,\W_{10}.
\ee
However, the only known almost contact metric structure of class $\W_4\op\W_7$ is neither of class $\W_4$ nor of class $\W_7$, it is a so-called nearly cosymplectic structure constructed on the $5$-sphere (see \cite{Bla71}). Examples with
\be
\Gamma\in\W_4,\quad \Gamma\neq0
\ee
are not only interesting in order to complete the list of examples in \cite{Puh12}. They would contradict proposition 3.1 in \cite{FI03} (cf.\ proposition \ref{prop:1} and remark \ref{rmk:1}). Moreover, structures of this type would admit a totally skew-symmetric Nijenhuis tensor $N\neq0$ (see theorem \ref{thm:1} together with propositions \ref{prop:1} and \ref{prop:2}).

From \autoref{sec:4} onward, we concentrate on almost contact metric structures of class $\W_4\op\W_7$. These generalized quasi-Sasaki manifolds are not quasi-Sasaki and can be characterized (see theorem \ref{thm:2}) by
\be
N(\xi,X,Y)=2\cdot d\eta(X,Y).
\ee
Moreover, we determine the following invariance property of the exterior derivative $d\eta$ (see proposition \ref{prop:3}):
\be
d\eta(\vphi(X),\vphi(Y))=-d\eta(X,Y).
\ee
This property enables us to construct a family $M^5(a_1,a_2,a_3,a_4)$ of almost contact metric $5$-manifolds of class $\W_4\op\W_7$ depending on four real parameters $a_1,a_2,a_3,a_4$, $a_1a_4=a_2a_3$ (see \autoref{sec:5}). These are submanifolds of $6$-dimensional Lie groups $G(a_1,a_2,a_3,a_4)$. We are able to identify the latter in those situations where at least one of the parameters $a_1,\ldots,a_4$ is zero (see propositions \ref{prop:7}, \ref{prop:6} and \ref{prop:4} together with remark \ref{rmk:3}). For example, if $a_1^2+a_2^2\neq0$ and $a_3=a_4=0$, $G(a_1,a_2,a_3,a_4)$ is locally isomorphic to $S^3\x S^3$ (see proposition \ref{prop:7}). As a result, the $5$-dimensional Stiefel manifold (see \cite{Sti35}) admits an almost contact metric structure of class $\W_4$ that is not cosymplectic (see example \ref{exa:1}).

Also, we study the metric connection $\nabla^c$ of the generalized quasi-Sasaki manifold $M^5(a_1,a_2,a_3,a_4)$. In this case, the holonomy group of $\nabla^c$ is $\Un(1)$ (see theorem \ref{thm:3},c)). Consequently, there exist two spinor fields that are parallel with respect to $\nabla^c$ (see theorem \ref{thm:3},e)).
%
%
%
%
\section{Almost contact metric \texorpdfstring{$5$}{5}-manifolds}\label{sec:2}\noindent
We first introduce some notation. Let $\lb(M^{2k+1},g\rb)$ be a Riemannian manifold of dimension $2k+1$. An \emph{almost contact metric structure} on $\lb(M^{2k+1},g\rb)$ consists of a vector field $\xi$ of length one, its dual $1$-form $\eta$ and an endomorphism $\vphi$ of the tangent bundle such that
\ba
\vphi\lb(\xi\rb)&=0, & \vphi^2&=-\Id+\eta\ox\xi, &
g\lb(\vphi\lb(X\rb),\vphi\lb(Y\rb)\rb)=g\lb( X,Y\rb)-\eta\lb(X\rb)\eta\lb(Y\rb).
\ea
Equivalently, these structures can be defined as a reduction of the structure group of orthonormal frames of the tangent bundle to the Lie group $\Un\left(k\right)$ (see \cites{Gra59,SH61,SH62}). The \emph{fundamental form} $\Phi$ of an \emph{almost contact metric manifold} $\lb(M^{2k+1},g,\xi,\eta,\vphi\rb)$ is a $2$-form given by
\be
\Phi\lb(X,Y\rb):=g\lb(X,\vphi\lb(Y\rb)\rb). 
\ee
In the exterior algebra, $\Phi$ satisfies $\eta\wedge\Phi^k\neq0$. Consequently, there exists an oriented orthonormal frame $\left(e_1,\ldots,e_{2k+1}\right)$ realizing
\ba
\Phi&=e_{1}\wedge e_{2}+\ldots+e_{2k-1}\wedge e_{2k}, & \eta&=e_{2k+1}
\ea
at every point of $M^{2k+1}$. Here and henceforth, we identify $TM^{2k+1}$ with its dual space using $g$. We call $\left(e_1,\ldots,e_{2k+1}\right)$ an \emph{adapted frame} of the almost contact metric manifold.  The connection forms
\be
\omega_{ij}^g:=g\lb(\nabla^{g}e_i,e_j\rb)
\ee
of the Levi-Civita connection $\nabla^{g}$ with respect to $\left(e_1,\ldots,e_{2k+1}\right)$ define a $1$-form 
\be
\Omega^g:=\lb(\omega_{ij}^g\rb)_{1\leq i,j\leq 2k+1}
\ee
with values in the Lie algebra $\so(2k+1)$. We define the \emph{intrinsic torsion} $\Gamma$ of the almost contact metric manifold as
\be
\Gamma(X) := \pr_{\un(k)^\perp}\lb(\Omega^g(X)\rb).
\ee
Here, $\pr_{\un(k)^\perp}$ denotes the projection onto the orthogonal complement $\un(k)^\perp$ of the Lie algebra $\un(k)$ inside $\lb(\so(2k+1),\lan x,y\ran=-(1/2)\cdot\tr\lb(x\circ y\rb)\rb)$,
\be
\so(2k+1) = \un(k)\op\un(k)^\perp.
\ee
Viewed as a $\left(3,0\right)$-tensor field, the \emph{Nijenhuis tensor} $N$ of $\lb(M^{2k+1},g,\xi,\eta,\vphi\rb)$ is defined by
\be
N\lb(X,Y,Z\rb):=g\lb(X,\lb[\vphi,\vphi\rb]\lb(Y,Z\rb)\rb)+\eta\lb(X\rb)d\eta\lb(Y,Z\rb), 
\ee
where $\lb[\vphi,\vphi\rb]$ is the Nijenhuis torsion of $\vphi$,
\ba
\lb[\vphi,\vphi\rb]\lb(X,Y\rb)=&
\lb[\vphi\lb(X\rb),\vphi\lb(Y\rb)\rb]
+\vphi^2\lb(\lb[X,Y\rb]\rb)
-\vphi\lb(\lb[\vphi\lb(X\rb),Y\rb]\rb)
-\vphi\lb(\lb[X,\vphi\lb(Y\rb)\rb]\rb)\\
=&\lb(\nabla^g_{\vphi\lb(X\rb)}\vphi\rb)\lb(Y\rb)
-\lb(\nabla^g_{\vphi\lb(Y\rb)}\vphi\rb)\lb(X\rb)
+\vphi\lb(\lb(\nabla^g_Y\vphi\rb)\lb(X\rb)-\lb(\nabla^g_X\vphi\rb)\lb(Y\rb)\rb),
\ea
and the differential $d\alpha$ and the codifferential $\delta\alpha$ of a differential form $\alpha$ are given by
\ba
d\alpha&=\sum_i e_i\wedge\nabla^{g}_{e_i}\alpha, &
\delta\alpha&=-\sum_i e_i\hook\nabla^{g}_{e_i}\alpha.
\ea
We conclude $N\in\Lambda^1\ox\Lambda^2$ denoting by $\Lambda^1$ and $\Lambda^2$ the spaces of $1$- and $2$-forms, respectively. The space
\be
\A:=\Lambda^1\ox\Lambda^2
\ee
splits into three irreducible $\Orth(2k+1)$-modules (see \cite{Car25}),
\be
\A=\A_1\op\A_2\op\A_3,
\ee
and a tensor field $A\in\A$ is said to be
\bite
\item\emph{vectorial} if $A\in\A_1$, or equivalently if there exists a vector field $V$ such that
\be
A(X,Y,Z)=g(X,Y)g(Z,V)-g(X,Z)g(Y,V).
\ee
\item\emph{totally skew-symmetric} if $A\in\A_2$, or equivalently if
\be
A(X,Y,Z)+A(Y,X,Z)=0.
\ee
\item\emph{traceless cyclic} if $A\in\A_3$, or equivalently if
\ba
A(X,Y,Z)+A(Y,Z,X)+A(Z,X,Y)&=0, & \sum_iA(e_i,e_i,X)&=0.
\ea
\eite

We then specialize to the $5$-dimensional case, i.e.\ $k=2$. In this situation, the space 
\be
\W:=\Lambda^1\ox\un(2)^\perp
\ee
of possible intrinsic torsion tensors splits into $10$ irreducible $\Un(2)$-modules (cf.\ \cite{Puh12}):
\be
\W=\W_1\op\ldots\op\W_{10}.
\ee
We say that an almost contact metric $5$-manifold \emph{is of class} $\W_{i_1}\op\ldots\op\W_{i_j}$ if
\be
\Gamma\in\W_{i_1}\op\ldots\op\W_{i_j}.
\ee
Moreover, $\lb(M^5,g,\xi,\eta,\vphi\rb)$ \emph{is of strict class} $\W_{i_1}\op\ldots\op\W_{i_j}$ if it is of class $\W_{i_1}\op\ldots\op\W_{i_j}$ and $\Gamma_{i_1}\neq0,\ldots,\Gamma_{i_j}\neq0$ denoting by $\Gamma_{i_1},\ldots,\Gamma_{i_j}$ the components of $\Gamma$ in $\W_{i_1},\ldots,\W_{i_j}$, respectively. Almost contact metric $5$-manifolds with $\Gamma=0$ are called \emph{integrable}. Consequently, there exist $2^{10}=1024$ classes according to this approach. Obviously, many of them neither have been studied nor carry any name. We present some of the established types of almost contact metric structures. An almost contact metric manifold is said to be
\bite
\item\emph{normal} (see \cites{SH61,SH62}) if its Nijenhuis tensor vanishes.
\item\emph{semi-cosymplectic} (see \cite{CG90}) if
\ba
\delta\Phi&=0, & \delta\eta&=0.
\ea
\item\emph{almost cosymplectic} (see \cite{Oub85}) if
\ba
d\Phi&=0, & d\eta&=0.
\ea
\item\emph{cosymplectic} (see \cite{Bla67}) if it is normal and almost cosymplectic, or equivalently if it is integrable.
\item\emph{quasi-Sasaki} (see \cite{Bla67}) if it is normal and 
\be
d\Phi=0.
\ee
\item\emph{nearly cosymplectic} (see \cite{Bla71}) if
\be
\lb(\nabla^g_X\vphi\rb)\lb(X\rb)=0.
\ee
\item\emph{quasi-cosymplectic} (see \cite{CM92}) if
\be
\lb(\nabla^g_X\vphi\rb)\lb(Y\rb)+\lb(\nabla^g_{\vphi\lb(X\rb)}\vphi\rb)\lb(\vphi\lb(Y\rb)\rb)=\eta\lb(Y\rb)\cdot\nabla^g_{\vphi\lb(X\rb)}\xi.
\ee
\eite
Moreover, we introduce those $\Un(2)$-submodules of $\W$ which are relevant for the purpose of this article. The Hodge operator $\ast$ is $\Un(2)$-equi\-var\-i\-ant, and the space of $2$-forms
\be
\Lambda^2=\Lambda^2_1\op\Lambda^2_2\op\Lambda^2_3\op\Lambda^2_4
\ee
decomposes into four irreducible $\Un(2)$-modules:
\ba
\Lambda^2_1 &:= \lb\{t\cdot\Phi\setsep t\in\R\rb\}, \\
\Lambda^2_2 &:= \lb\{\beta\in\Lambda^2\setsep \Phi\wedge\beta=0,\  \ast\beta=\eta\wedge\beta\rb\},\\
\Lambda^2_3 &:= \lb\{\beta\in\Lambda^2\setsep \ast\beta=-\eta\wedge\beta\rb\},\\
\Lambda^2_4 &:= \lb\{\beta\in\Lambda^2\setsep \eta\wedge\beta=0\rb\}.
\ea
The dimensions of these modules are
\be
\dim\lb(\Lambda^2_i\rb)=i,
\ee
and we identify $\so(5)$ with $\Lambda^2$,
\ba
\un(2)&=\Lambda^2_1\op\Lambda^2_3, & \un(2)^\perp &=\Lambda^2_2\op\Lambda^2_4.
\ea
With the aid of the $\Un(2)$-equivariant maps
\ba
\theta &:\Lambda^2\ra\A, &
\theta&\lb(\beta\rb)\lb(X,Y,Z\rb):=\ast\beta\lb(X,Y,Z\rb),\\
\vartheta &:\Lambda^2\ra\A, &
\vartheta&\lb(\beta\rb)\lb(X,Y,Z\rb):=3\,\eta\lb(X\rb)\beta\lb(Y,Z\rb)-\ast\beta\lb(X,Y,Z\rb)
\ea
and the projection $\pr_{\W}:\A\ra\W$,
\be
\pr_{\W}\lb(\alpha\ox\beta\rb):=\alpha\ox\pr_{\un(2)^\perp}\lb(\beta\rb),
\ee
we are able to describe the following irreducible $\Un(2)$-submodules of $\W$:
\ba
\W_3&=\pr_{\W}\lb(\theta\lb(\Lambda^2_1\rb)\rb), &
\W_4&=\pr_{\W}\lb(\theta\lb(\Lambda^2_2\rb)\rb)=\theta\lb(\Lambda^2_2\rb), & \W_5&=\pr_{\W}\lb(\theta\lb(\Lambda^2_3\rb)\rb),\\
\W_6&=\pr_{\W}\lb(\theta\lb(\Lambda^2_4\rb)\rb), &
\W_7&=\pr_{\W}\lb(\vartheta\lb(\Lambda^2_2\rb)\rb)=\vartheta\lb(\Lambda^2_2\rb).
\ea

The Riemannian covariant derivatives of $\xi$, $\eta$, $\vphi$ and $\Phi$, together with the Nijenhuis tensor $N$, can be expressed in terms of the intrinsic torsion (see \cite{Puh12}):
\be
g\lb(\nabla^g_X\xi,Y\rb)=\lb(\nabla^g_X\eta\rb)\lb(Y\rb)=\lb(\nabla^g_X\Phi\rb)\lb(\xi,\vphi\lb(Y\rb)\rb),
\ee
\be
g\lb(\lb(\nabla^g_X\vphi\rb)\lb(Y\rb),Z\rb)=\lb(\nabla^g_X\Phi\rb)\lb(Z,Y\rb),
\ee
\ba
N\lb(X,Y,Z\rb)=&\big(\nabla^g_{\vphi\lb(Y\rb)}\Phi\big)\lb(X,Z\rb)
-\big(\nabla^g_{\vphi\lb(Z\rb)}\Phi\big)\lb(X,Y\rb)
+\lb(\nabla^g_Y\Phi\rb)\lb(\vphi\lb(X\rb),Z\rb)\\
&-\lb(\nabla^g_Z\Phi\rb)\lb(\vphi\lb(X\rb),Y\rb)+\eta\lb(X\rb)\lb(\nabla^g_{Y}\Phi\rb)\lb(\xi,\vphi\lb(Z\rb)\rb)
-\eta\lb(X\rb)\lb(\nabla^g_{Z}\Phi\rb)\lb(\xi,\vphi\lb(Y\rb)\rb),
\ea
\be
\lb(\nabla^{g}_X\Phi\rb)\lb(Y,Z\rb) = \sum_i\Gamma\lb(X\rb)\lb(e_i,Y\rb)\Phi\lb(e_i,Z\rb)-\Gamma\lb(X\rb)\lb(e_i,Z\rb)\Phi\lb(e_i,Y\rb).
\ee
We denote by ($\bullet$) the system of equations above. Using this system, one can immediately check the following proposition.
\begin{prop}\label{prop:1}
An almost contact metric $5$-manifold $\lb(M^5,g,\xi,\eta,\vphi\rb)$ is of class
\bite
\item[a)] $\W_3\op\W_4\op\W_5\op\W_6$ if and only if its Nijenhuis tensor $N$ is totally skew-symmetric and $\xi$ is a Killing vector field.
\item[b)] $\W_3\op\W_5\op\W_6$ if and only if its Nijenhuis tensor $N$ is totally skew-symmetric and $\xi\hook d\Phi=0$.
\eite
\end{prop}
\begin{rmk}\label{rmk:1}
Consequently, any non-integrable ($\Gamma\neq0$) example of class $\W_4$ would contradict proposition 3.1 in \cite{FI03}.
\end{rmk}
%
%
%
%
%
\section{Generalized quasi-Sasaki manifolds}\label{sec:3}\noindent
Introduced as an attempt to unify Sasakian and cosymplectic geometry (see \cite{Bla67}), quasi-Sasaki manifolds $\lb(M^5,g,\xi,\eta,\vphi\rb)$ satisfy
\ba
N&=0,& d\Phi&=0,
\ea
and $\xi$ is a Killing vector field with respect to $g$. Motivated by these three properties, we define a \emph{generalized quasi-Sasaki manifold} $\lb(M^5,g,\xi,\eta,\vphi\rb)$ as an almost contact metric $5$-manifold such that the following hold:
\bite
\item[i)] The equations
\ba
N(X,Y,Z)&=0,& d\Phi(X,Y,Z)&=0
\ea
are satisfied for any $X,Y,Z$ orthogonal to $\xi$.
\item[ii)] The vector field $\xi$ is a Killing vector field with respect to $g$.
\eite
\begin{prop}\label{prop:5}
An almost contact metric $5$-manifold is a generalized quasi-Sasaki manifold (resp.\ quasi-Sasaki manifold) if and only if it is of class $\W_3\op\W_4\op\W_5\op\W_7$ (resp.\ class $\W_3\op\W_5$).
\end{prop}
\begin{proof}
The proof of the second statement can be found in \cite{Puh12}. Let $\lb(M^5,g,\xi,\eta,\vphi\rb)$ be a generalized quasi-Sasaki manifold. Then, we have
\ba
g\lb(\nabla^g_X\xi,Y\rb)+g\lb(\nabla^g_Y\xi,X\rb)&=0, & N\lb(\vphi(X),\vphi(Y),\vphi(Z)\rb)&=0,& d\Phi\lb(\vphi(X),\vphi(Y),\vphi(Z)\rb)&=0.
\ea
With the aid of system ($\bullet$), we compute
\be
g\lb(\nabla^g_X\xi,Y\rb)+g\lb(\nabla^g_Y\xi,X\rb)=\lb(\nabla^g_X\Phi\rb)\lb(\xi,\vphi\lb(Y\rb)\rb)+\lb(\nabla^g_Y\Phi\rb)\lb(\xi,\vphi\lb(X\rb)\rb),
\ee
\ba
N\lb(\vphi(X),\vphi(Y),\vphi(Z)\rb)=&\big(\nabla^g_{\vphi^2\lb(Y\rb)}\Phi\big)\lb(\vphi(X),\vphi(Z)\rb)
-\big(\nabla^g_{\vphi^2\lb(Z\rb)}\Phi\big)\lb(\vphi(X),\vphi(Y)\rb)\\
&+\lb(\nabla^g_{\vphi(Y)}\Phi\rb)\lb(\vphi^2\lb(X\rb),\vphi(Z)\rb)
-\lb(\nabla^g_{\vphi(Z)}\Phi\rb)\lb(\vphi^2\lb(X\rb),\vphi(Y)\rb)\\
=&-\big(\nabla^g_{Y}\Phi\big)\lb(\vphi(X),\vphi(Z)\rb)
+\big(\nabla^g_{Z}\Phi\big)\lb(\vphi(X),\vphi(Y)\rb)\\
&-\lb(\nabla^g_{\vphi(Y)}\Phi\rb)\lb(X,\vphi(Z)\rb)
+\lb(\nabla^g_{\vphi(Z)}\Phi\rb)\lb(X,\vphi(Y)\rb)\\
&+\eta(Y)\big(\nabla^g_{\xi}\Phi\big)\lb(\vphi(X),\vphi(Z)\rb)
-\eta(Z)\big(\nabla^g_{\xi}\Phi\big)\lb(\vphi(X),\vphi(Y)\rb)\\
&+\eta(X)\lb(\nabla^g_{\vphi(Y)}\Phi\rb)\lb(\xi,\vphi(Z)\rb)
-\eta(X)\lb(\nabla^g_{\vphi(Z)}\Phi\rb)\lb(\xi,\vphi(Y)\rb)
\ea
and
\ba
d\Phi\lb(\vphi(X),\vphi(Y),\vphi(Z)\rb)=&\lb(\nabla^g_{\vphi(X)}\Phi\rb)\lb(\vphi(Y),\vphi(Z)\rb)-\lb(\nabla^g_{\vphi(Y)}\Phi\rb)\lb(\vphi(X),\vphi(Z)\rb)\\&+\lb(\nabla^g_{\vphi(Z)}\Phi\rb)\lb(\vphi(X),\vphi(Y)\rb).
\ea
This, together with
\be
\lb(\nabla^{g}_X\Phi\rb)\lb(Y,Z\rb) = \sum_i\Gamma\lb(X\rb)\lb(e_i,Y\rb)\Phi\lb(e_i,Z\rb)-\Gamma\lb(X\rb)\lb(e_i,Z\rb)\Phi\lb(e_i,Y\rb),
\ee
completes the proof, since all of our steps are reversible.
\end{proof}
\begin{exa}
In \cite{FI03}, the authors construct a family of quasi-Sasaki $5$-manifolds depending on four real parameters $a,b,c,d$. These are submanifolds of $6$-dimensional Lie groups including $S^3\times S^3$. From the classification viewpoint above, this family realizes class $\W_3$ if $a=d$ and $b=c=0$, class $\W_5$ if $a=-d$, and strict class $\W_3\op\W_5$ if $a\neq-d$ and $(a-d)^2+b^2+c^2\neq0$. The only known almost contact metric structure of class $\W_4\op\W_7$ is of strict class $\W_4\op\W_7$, it is a nearly cosymplectic structure constructed on the $5$-sphere (see \cite{Bla71}).
\end{exa}
We investigate connections with torsion on generalized quasi-Sasaki manifolds. Let $\nabla$ be a metric connection on an almost contact metric $5$-manifold $\lb(M^5,g,\xi,\eta,\vphi\rb)$, i.e.\
\be
g\lb(\nabla_XY,Z\rb)=g\lb(\nabla^{g}_XY,Z\rb)+A\lb(X,Y,Z\rb)
\ee
for $A\in\A$. Its torsion tensor $T$, viewed as a $\lb(3,0\rb)$-tensor field, is given by
\ba
T\lb(X,Y,Z\rb):=&g\lb(\nabla_XY-\nabla_YX-\lb[X,Y\rb],Z\rb)\\
=&A\lb(X,Y,Z\rb)-A\lb(Y,X,Z\rb).
\ea
Consequently, we have $T\in\Lambda^2\ox\Lambda^1$. We say that $T$ is
\bite
\item \emph{vectorial} if $A$ is vectorial, or equivalently if there exists a vector field $V$ such that
\be
T(X,Y,Z)=g(X,V)g(Y,Z)-g(Y,V)g(X,Z).
\ee
\item \emph{totally skew-symmetric} if $A$ is totally skew-symmetric, or equivalently if
\be
T(X,Y,Z)+T(X,Z,Y)=0.
\ee
\item \emph{traceless cyclic} if $A$ is traceless cyclic, or equivalently if
\ba
T(X,Y,Z)+T(Y,Z,X)+T(Z,X,Y)&=0, & \sum_iT(X,e_i,e_i)&=0.
\ea
\eite
Moreover, $\nabla$ is said to be \emph{compatible} with the underlying almost contact metric structure if
\ba
\nabla\xi&=0, &
\nabla\eta&=0, &
\nabla\vphi&=0.
\ea
\begin{thm}\label{thm:1}
Any generalized quasi-Sasaki manifold $\lb(M^5,g,\xi,\eta,\vphi\rb)$ admits a unique metric connection $\nabla^c$ that is compatible with the underlying almost contact metric structure. The connection $\nabla^c$ satisfies
\be
g\lb(\nabla^c_XY,Z\rb)=g\lb(\nabla^{g}_XY,Z\rb)+\frac{1}{2}\lb\{\lb(\lb(d\eta-\gamma\rb)\wedge\eta\rb)(X,Y,Z)-N(X,Y,Z)\rb\},
\ee
where $\gamma$ is a $2$-form given by
\be
\gamma(X,Y):=d\Phi\lb(\xi,\vphi(X),Y\rb)=N\lb(\vphi(X),\vphi(Y),\xi\rb).
\ee
Moreover, the torsion tensor of $\nabla^c$ is
\bite
\item[a)] totally skew-symmetric if and only if $N$ is totally skew-symmetric. In this case,
\be
N(X,Y,Z)+\lb(\gamma\wedge\eta\rb)(X,Y,Z)=0.
\ee
\item[b)] traceless cyclic if and only if $d\eta=\gamma$. In this case, $N$ is traceless cyclic.
\eite
\end{thm}
\begin{proof}
Let $\lb(M^5,g,\xi,\eta,\vphi\rb)$ be a generalized quasi-Sasaki manifold. Since $\xi$ is a Killing vector field, we have
\be
d\eta(\xi,X)=\big(\nabla^g_\xi\eta\big)(X)-\big(\nabla^g_X\eta\big)(\xi)=g\big(\nabla^g_\xi\xi,X\big)-g\big(\nabla^g_X\xi,\xi\big)=-2\, g\big(\nabla^g_X\xi,\xi\big)=0.
\ee
As a result, we obtain (cf.\ \cite{Bla76}*{chapter IV, equation (6)})
\ba
2\, g\lb(\lb(\nabla^g_\xi\vphi\rb)(Y),\vphi(X)\rb)=&d\Phi\lb(\xi,\vphi(Y),\vphi^2(X)\rb)-d\Phi\lb(\xi,Y,\vphi(X)\rb)+\lb(\Lie_{\vphi(Y)}\eta\rb)\lb(\vphi(X)\rb)\\
&-\lb(\Lie_{\vphi^2(X)}\eta\rb)\lb(Y\rb)\\
=&d\Phi\lb(\xi,X,\vphi(Y)\rb)-d\Phi\lb(\xi,Y,\vphi(X)\rb)+d\eta\lb(\vphi(Y),\vphi(X)\rb)\\
&+d\eta\lb(X,Y\rb).
\ea
Consequently,
\be
g\lb(\lb(\nabla^g_\xi\vphi\rb)(Y),\vphi(X)\rb)+g\lb(\lb(\nabla^g_\xi\vphi\rb)(X),\vphi(Y)\rb)=0.
\ee
This, together with
\ba
d\Phi\lb(\xi,\vphi(X),Y\rb)=&\lb(\nabla^g_\xi\Phi\rb)\lb(\vphi(X),Y\rb)-\lb(\nabla^g_{\vphi(X)}\Phi\rb)\lb(\xi,Y\rb)+\lb(\nabla^g_Y\Phi\rb)\lb(\xi,\vphi(X)\rb)\\
=&g\lb(\lb(\nabla^g_\xi\vphi\rb)(Y),\vphi(X)\rb)+g\lb(\nabla^g_{\vphi(X)}\xi,\vphi(Y)\rb)+g\lb(\nabla^g_Y\xi,X\rb),
\ea
proves that $\gamma$ is a $2$-form, and
\ba
N\lb(\vphi(X),\vphi(Y),\xi\rb)=&g\lb(\vphi(X),\vphi^2\lb(\lb[\vphi(Y),\xi\rb]\rb)-\vphi\lb(\lb[\vphi^2(Y),\xi\rb]\rb)\rb)\\
=&g\lb(\vphi(X),-\lb[\vphi(Y),\xi\rb]+\vphi\lb([Y,\xi]\rb)\rb)\\
=&g\lb(\vphi(X),-\nabla^g_{\vphi(Y)}\xi-\lb(\nabla^g_Y\vphi\rb)(\xi)+\lb(\nabla^g_\xi\vphi\rb)(Y)\rb)\\
=&g\lb(\nabla^g_{\vphi(X)}\xi,\vphi(Y)\rb)+g\lb(\nabla^g_Y\xi,X\rb)+g\lb(\lb(\nabla^g_\xi\vphi\rb)(Y),\vphi(X)\rb)\\
=&d\Phi\lb(\xi,\vphi(X),Y\rb).
\ea
We now define the metric connection $\nabla^c$ via
\ba
g\lb(\nabla^c_XY,Z\rb)&:=g\lb(\nabla^{g}_XY,Z\rb)+A^c(X,Y,Z),\\
A^c(X,Y,Z)&:=\frac{1}{2}\lb\{\lb(\lb(d\eta-\gamma\rb)\wedge\eta\rb)(X,Y,Z)-N(X,Y,Z)\rb\}.
\ea
The connection forms
\be
\omega^c_{ij}:=g\lb(\nabla^c e_i,e_j\rb)
\ee
of $\nabla^c$ define a $1$-form
\be
\Omega^c:=\lb(\omega^c_{ij}\rb)_{1\leq i,j\leq 5}
\ee
with values in the Lie algebra $\so(5)$. We project onto $\un(2)^\perp$:
\be
\pr_{\un(2)^\perp}\lb(\Omega^c\lb(X\rb)\rb)=\Gamma\lb(X\rb)+\pr_{\W}\lb(A^c\rb)\lb(X\rb).\\
\ee
A direct computation involving system ($\bullet$) verifies
\be
\Gamma+\pr_{\W}\lb(A^c\rb)=0
\ee
and $A^c\in\theta\lb(\Lambda^2\rb)\op\vartheta\lb(\Lambda^2_2\rb)$ (cf.\ \autoref{sec:2}). Therefore, $\nabla^c$ is compatible with the underlying almost contact metric structure. The uniqueness of $\nabla^c$ is a consequence of the injectivity of $\pr_{\W}$ when restricted to $\theta\lb(\Lambda^2\rb)\op\vartheta\lb(\Lambda^2_2\rb)$. Any tensor field in $\theta\lb(\Lambda^2\rb)$ (resp.\ $\vartheta\lb(\Lambda^2_2\rb)$) is totally skew-symmetric (resp.\ traceless cyclic). Since $A^c\in\vartheta\lb(\Lambda^2_2\rb)$ if and only if $d\eta=\gamma$, we deduce b). Finally, a) follows immediately by theorem 8.2 in \cite{FI02}, since
\be
d\Phi\lb(\vphi(X),\vphi(Y),\vphi(Z)\rb)=0,
\ee
\be
N(X,Y,Z)=\eta(X)N(\xi,Y,Z)-\eta(Y)N(\xi,X,Z)+\eta(Z)N(\xi,X,Y)=\lb(\eta\wedge\lb(\xi\hook N\rb)\rb)(X,Y,Z)
\ee
hold for any generalized quasi-Sasaki manifold with totally skew-symmetric Nijenhuis tensor $N$.
\end{proof}
Revisiting the proof of theorem \ref{thm:1}, we conclude the following proposition.
\begin{prop}\label{prop:2}
The torsion tensor of the connection $\nabla^c$ is totally skew-symmetric (resp.\ traceless cyclic) if and only if the generalized quasi-Sasaki manifold is of class $\W_3\op\W_4\op\W_5$ (resp.\ class $\W_7$).
\end{prop}
\begin{rmk}\label{rmk:2}
Note that the class of quasi-Sasaki manifolds (i.e.\ the class $\W_3\op\W_5$) coincides with the class of normal generalized quasi-Sasaki manifolds (cf.\ proposition \ref{prop:1}). Therefore, any non-integrable example of class $\W_4$ (resp.\ class $\W_7$) would admit a non-trivial Nijenhuis tensor that is totally skew-symmetric (resp.\ traceless cyclic).
\end{rmk}
%
%
%
%
%
\section{The class \texorpdfstring{$\W_4\op\W_7$}{W4+W7}}\label{sec:4}\noindent
From now on, we consider almost contact metric $5$-manifolds of class $\W_4\op\W_7$. To begin with, we ask how these generalized quasi-Sasaki manifolds relate to the special types of almost contact metric manifolds introduced in \autoref{sec:2}.
\begin{prop}
Let $\lb(M^5,g,\xi,\eta,\vphi\rb)$ be a non-integrable almost contact metric $5$-ma\-ni\-fold of class $\W_4\op\W_7$. Then the following hold:
\bite
\item[a)] $\lb(M^5,g,\xi,\eta,\vphi\rb)$ is semi-cosymplectic.
\item[b)] $\lb(M^5,g,\xi,\eta,\vphi\rb)$ is neither normal nor almost cosymplectic.
\item[c)] If $\lb(M^5,g,\xi,\eta,\vphi\rb)$ is nearly cosymplectic or quasi-cosymplectic, then it is of strict class $\W_4\op\W_7$.
\eite
\end{prop}
\begin{proof}
The results follow directly from the definitions and system ($\bullet$).
\end{proof}
We then characterize the class $\W_4\op\W_7$ in terms of differential equations.
\begin{thm}\label{thm:2}
An almost contact metric $5$-manifold $\lb(M^5,g,\xi,\eta,\vphi\rb)$ is of class $\W_4\op\W_7$ if and only if it is a generalized quasi-Sasaki manifold such that 
\be
N(\xi,X,Y)=2\cdot d\eta(X,Y).
\ee
Moreover, any almost contact metric $5$-manifold of class $\W_4\op\W_7$ is of class
\bite
\item[a)] $\W_4$ if and only if $N$ is totally skew-symmetric, or equivalently if $\gamma=-2\cdot d\eta$. In this case,
\be
N(X,Y,Z)=2\lb(d\eta\wedge\eta\rb)(X,Y,Z).
\ee
\item[b)] $\W_7$ if and only if $N$ is traceless cyclic, or equivalently if $\gamma=d\eta$. In this case,
\be
N(X,Y,Z)=2\,\eta(X)d\eta(Y,Z)+\eta(Y)d\eta(X,Z)-\eta(Z)d\eta(X,Y).
\ee
\eite
\end{thm}
\begin{proof}
The first statement is a direct consequence of system ($\bullet$). Let $\lb(M^5,g,\xi,\eta,\vphi\rb)$ be an almost contact metric $5$-manifold of class $\W_4\op\W_7$. Since
\be\tag{$\ast$}
N(\xi,X,Y)=2\cdot d\eta(X,Y)
\ee
and $\xi\hook d\eta=0$ hold, we have
\ba
\gamma\lb(\vphi(X),\vphi(Y)\rb)&=d\Phi\lb(\xi,\vphi^2(X),\vphi(Y)\rb)=d\Phi\lb(\xi,\vphi(Y),X\rb)=-\gamma(X,Y),\\
\gamma\lb(\vphi(X),\vphi(Y)\rb)&=N\lb(\vphi^2(X),\vphi^2(Y),\xi\rb)=-N\lb(X,\vphi^2(Y),\xi\rb)=N\lb(X,Y,\xi\rb).
\ea
Therefore,
\be\tag{$\ast\ast$}
N(X,Y,\xi)=-\gamma(X,Y),
\ee
and $N$ is totally skew-symmetric if and only if
\be
N(\xi,X,Y)=-\gamma(X,Y).
\ee
By ($\ast$), the latter is equivalent to $\gamma=-2\cdot d\eta$. Theorem \ref{thm:1},a) and
proposition \ref{prop:2} complete the proof of a). Suppose $N$ is traceless cyclic. Using ($\ast$) and ($\ast\ast$), we compute  
\be
2\cdot d\eta(X,Y)=N(\xi,X,Y)=-N(X,Y,\xi)+N(Y,X,\xi)=2\cdot\gamma(X,Y),
\ee
hence $\gamma=d\eta$. The remaining statements of b) result from theorem \ref{thm:1},b) and
proposition \ref{prop:2}.
\end{proof}
On quasi-Sasaki manifolds, the exterior derivative $d\eta$ has the following invariance property (see \cite{Bla76}*{chapter IV, equation (5)}):
\be
d\eta(\vphi(X),\vphi(Y))=d\eta(X,Y).
\ee
The situation is similar in the case of class $\W_4\op\W_7$.
\begin{prop}\label{prop:3}
Let $\lb(M^5,g,\xi,\eta,\vphi\rb)$ be an almost contact metric $5$-manifold of class $\W_4\op\W_7$. Then
\be
d\eta(\vphi(X),\vphi(Y))=-d\eta(X,Y).
\ee
\end{prop}
\begin{proof}
With the aid of theorem \ref{thm:2}, we compute
\ba
d\eta(X,Y)=&N(\xi,X,Y)-d\eta(X,Y)\\
=&g\lb(\xi,[\vphi,\vphi](X,Y)\rb)\\
=&g\lb(\xi,\lb(\nabla^g_{\vphi\lb(X\rb)}\vphi\rb)\lb(Y\rb)
-\lb(\nabla^g_{\vphi\lb(Y\rb)}\vphi\rb)\lb(X\rb)
+\vphi\lb(\lb(\nabla^g_Y\vphi\rb)\lb(X\rb)-\lb(\nabla^g_X\vphi\rb)\lb(Y\rb)\rb)\rb)\\
=&g\lb(\xi,\nabla^g_{\vphi\lb(X\rb)}\lb(\vphi(Y)\rb)-\vphi\lb(\nabla^g_{\vphi\lb(X\rb)}Y\rb)
-\nabla^g_{\vphi\lb(Y\rb)}\lb(\vphi(X)\rb)+\vphi\lb(\nabla^g_{\vphi\lb(Y\rb)}X\rb)\rb)\\
=&-g\lb(\nabla^g_{\vphi\lb(X\rb)}\xi,\vphi(Y)\rb)+g\lb(\nabla^g_{\vphi\lb(Y\rb)}\xi,\vphi(X)\rb)\\
=&-d\eta(\vphi(X),\vphi(Y)).
\ea
\end{proof}
\begin{rmk}
In dimension $5$, any $2$-form $\beta$ satisfies
\be
\beta\lb(\vphi\lb(X\rb),\vphi\lb(Y\rb)\rb)=\begin{cases}
  \beta\lb(X,Y\rb) & \text{iff $\beta\in\Lambda^2_1\op\Lambda^2_3$}\\
  -\beta\lb(X,Y\rb) & \text{iff $\beta\in\Lambda^2_2$}\\
  0 & \text{iff $\beta\in\Lambda^2_4$}\ .
\end{cases}
\ee
\end{rmk}
We now consider an adapted frame $\lb(e_1,\ldots,e_5\rb)$ of an almost contact metric $5$-manifold $\lb(M^5,g,\xi,\eta,\vphi\rb)$ of class $\W_4\op\W_7$,
\ba
\Phi&=e_{1}\wedge e_{2}+e_{3}\wedge e_{4}, & \eta&=e_{5}.
\ea
Applying theorem \ref{thm:2} and proposition \ref{prop:3}, the structure equations yield
\ba
de_1 =& A_1\wedge e_2+A_2\wedge e_2+A_3\wedge e_3+A_4\wedge e_4-(2a_1+a_3)e_3\wedge e_5-(2a_2+a_4)e_4\wedge e_5,\\
de_2 =& -A_1\wedge e_1-A_2\wedge e_1+A_3\wedge e_4-A_4\wedge e_3+(2a_1+a_3)e_4\wedge e_5-(2a_2+a_4)e_3\wedge e_5,\\
de_3 =& A_1\wedge e_4-A_2\wedge e_4-A_3\wedge e_1+A_4\wedge e_2+(2a_1+a_3)e_1\wedge e_5+(2a_2+a_4)e_2\wedge e_5,\\
de_4 =& -A_1\wedge e_3+A_2\wedge e_3-A_3\wedge e_2-A_4\wedge e_1-(2a_1+a_3)e_2\wedge e_5+(2a_2+a_4)e_1\wedge e_5,\\
de_5 =& -2(a_1-a_3)(e_1\wedge e_3-e_2\wedge e_4)-2(a_2-a_4)(e_1\wedge e_4+e_2\wedge e_3)
\ea
for some $1$-forms $A_1,\ldots,A_4$ and functions $a_1,\ldots,a_4$. We discuss the integrability of this system in the case where $a_1,\ldots,a_4$ are constant and $A_1=A_3=A_4=0$. The corresponding system,
\ba
de_1 =& A_2\wedge e_2-(2a_1+a_3)e_3\wedge e_5-(2a_2+a_4)e_4\wedge e_5,\\
de_2 =& -A_2\wedge e_1+(2a_1+a_3)e_4\wedge e_5-(2a_2+a_4)e_3\wedge e_5,\\
de_3 =& -A_2\wedge e_4+(2a_1+a_3)e_1\wedge e_5+(2a_2+a_4)e_2\wedge e_5,\\
de_4 =& A_2\wedge e_3-(2a_1+a_3)e_2\wedge e_5+(2a_2+a_4)e_1\wedge e_5,\\
de_5 =& -2(a_1-a_3)(e_1\wedge e_3-e_2\wedge e_4)-2(a_2-a_4)(e_1\wedge e_4+e_2\wedge e_3),
\ea
will be denoted by ($\star$). A straightforward computation leads to the following lemma.
\begin{lem}
Let $a_1,\ldots,a_4$ be constants and $A_2$ be a $1$-form. Then, system ($\star$) is integrable if and only if
$a_1a_4=a_2a_3$ and
\be
dA_2=-2\big((a_1-a_3)(2a_1+a_3)+(a_2-a_4)(2a_2+a_4)\big)\lb(e_{1}\wedge e_{2}-e_{3}\wedge e_{4}\rb).
\ee
\end{lem}
%
%
%
%
\section{Examples}\label{sec:5}\noindent
We fix real numbers $a_1,\ldots,a_4$ such that $a_1a_4=a_2a_3$. Applying Lie's third theorem (cf.\ \cite{Fla63}), system ($\star$), together with
\be
dA_2=\alpha\cdot F,
\ee
\ba
\alpha&:=-2\big((a_1-a_3)(2a_1+a_3)+(a_2-a_4)(2a_2+a_4)\big), & F&:=\lb(e_{1}\wedge e_{2}-e_{3}\wedge e_{4}\rb),
\ea
has a solution in the Euclidean space $\R^6$. Moreover, there exists a $6$-dimensional Lie group $G(a_1,a_2,a_3,a_4)$ such that $e_1,\ldots,e_5,A_2$ constitute a basis of the left invariant $1$-forms. Let $M^5$ be a submanifold of $G(a_1,a_2,a_3,a_4)$ such that the restricted forms $e_1,\ldots,e_5$ are linearly independent. On $M^5$, we define a Riemannian metric and an almost contact metric structure by the condition that $(e_1,\ldots,e_5)$ is an adapted frame. The corresponding almost contact metric manifold will be denoted by $M^5(a_1,a_2,a_3,a_4)$. The connection forms of the Levi-Civita connection with respect to $(e_1,\ldots,e_5)$ are
\ba
\omega^g_{12}&=A_2, & \omega^g_{13}&=(a_1+2a_3) e_5, & \omega^g_{14}&=(a_2+2a_4) e_5,\\
\omega^g_{23}&=(a_2+2a_4) e_5, & \omega^g_{24}&=-(a_1+2a_3) e_5, & \omega^g_{34}&=-A_2,\\
\ea
\ba
\omega^g_{15}&=-(a_1-a_3)e_3-(a_2-a_4)e_4, & \omega^g_{25}&=-(a_2-a_4)e_3+(a_1-a_3)e_4, \\ 
\omega^g_{35}&=(a_1-a_3)e_1+(a_2-a_4)e_2, & \omega^g_{45}&=(a_2-a_4)e_1-(a_1-a_3)e_2.
\ea
A direct computation involving these forms verifies that $M^5(a_1,a_2,a_3,a_4)$ is a generalized quasi-Sasaki manifold satisfying
\be
N(\xi,X,Y)=2\cdot d\eta(X,Y).
\ee
As a consequence of theorem \ref{thm:1}, there exists a unique metric connection $\nabla^c$ that is compatible with the underlying almost contact metric structure. The $1$-form $A_2$ determines $\nabla^c$ completely:
\ba
\nabla^c_Xe_1&=A_2(X)\cdot e_2, & \nabla^c_Xe_2&=-A_2(X)\cdot e_1, \\
\nabla^c_Xe_3&=-A_2(X)\cdot e_4, & \nabla^c_Xe_4&=A_2(X)\cdot e_3, & \nabla^c e_5&=0.
\ea
In order to identify the algebraic types of both $N$ and the torsion tensor of $\nabla^c$ (cf.\ theorem \ref{thm:2} and proposition \ref{prop:2}), we compute
\ba
\gamma+2\cdot d\eta &=6a_3\cdot Z_1+6a_4\cdot Z_2, & \gamma-d\eta &=6a_1\cdot Z_1+6a_2\cdot Z_2,
\ea
\ba
Z_1&:=e_1\wedge e_3-e_2\wedge e_4, & Z_2&:=e_1\wedge e_4+e_2\wedge e_3.
\ea
We summarize our results as follows.
\begin{thm}\label{thm:3}
The almost contact metric manifold $M^5(a_1,a_2,a_3,a_4)$ is a generalized quasi-Sasaki manifold. In addition, $M^5(a_1,a_2,a_3,a_4)$ has the following properties:
\bite
\item[a)] The Nijenhuis tensor $N$ satisfies
\be
N(\xi,X,Y)=2\cdot d\eta(X,Y).
\ee
Moreover, $N$ is totally skew-symmetric (resp.\ traceless cyclic) if and only if
\be
a_3=a_4=0\quad\text{(resp.\ $a_1=a_2=0$)}.
\ee
\item[b)] The $2$-form $F$ is closed.
\item[c)] The curvature tensor $R^c:\Lambda^2\ra\un(2)$ of the connection $\nabla^c$ is proportional to the projection onto $\lb\{t\cdot F\setsep t\in\R\rb\}$:
\be
R^c=\alpha\cdot F\ox F.
\ee
\item[d)] The Ricci tensor of $\nabla^c$ is
\be
\Ric^c=-\alpha\cdot\diag(1,1,1,1,0).
\ee
\item[e)] There exist two $\nabla^c$-parallel spinor fields in the spin subbundle defined by the Clifford product $F\cdot\Psi=0$.
\eite
\end{thm}
\begin{cor}
The almost contact metric $5$-manifold $M^5(a_1,a_2,a_3,a_4)$ is of class
\be
\W_4\op\W_7.
\ee
Moreover, $M^5(a_1,a_2,a_3,a_4)$ is of class $\W_4$ (resp.\ class $\W_7$) if and only if $a_3=a_4=0$ (resp.\ $a_1=a_2=0$).
\end{cor}
Since
\ba
g\lb(\lb(\nabla^g_\xi\vphi\rb)(X),Y\rb)&=\lb(-2a_2-4a_4\rb)Z_1(X,Y)+\lb(2a_1+4a_3\rb)Z_2(X,Y),\\
g\lb(\lb(\nabla^g_X\vphi\rb)(\xi),Y\rb)&=\lb(a_2-a_4\rb)Z_1(X,Y)-\lb(a_1-a_3\rb)Z_2(X,Y)
\ea
and
\be
g\lb(\lb(\nabla^g_{\vphi(X)}\vphi\rb)(\vphi(Y)),\vphi(Z)\rb)=0,
\ee
we deduce the following proposition.
\begin{prop}
The almost contact metric $5$-manifold $M^5(a_1,a_2,a_3,a_4)$ is nearly cosymplectic (resp.\ quasi-cosymplectic) if and only if $a_1=-5a_3$ and $a_2=-5a_4$ (resp.\ $a_1=-2a_3$ and $a_2=-2a_4$).
\end{prop}
We now identify the Lie group $G(a_1,a_2,a_3,a_4)$ in those situations where at least one of the parameters $a_1,\ldots,a_4$ is zero, but
\be
a_1^2+a_2^2+a_3^2+a_4^2\neq0.
\ee
By definition, $a_1a_4=a_2a_3$. Consequently, we consider the following four cases:
\bite
\item[i)] $a_1^2+a_2^2\neq0$ and $a_3=a_4=0$.
\item[ii)] $a_3^2+a_4^2\neq0$ and $a_1=a_2=0$.
\item[iii)] $a_1^2+a_3^2\neq0$ and $a_2=a_4=0$.
\item[iv)] $a_2^2+a_4^2\neq0$ and $a_1=a_3=0$.
\eite
\begin{prop}\label{prop:7}
If $a_1^2+a_2^2\neq0$, the Lie group $G(a_1,a_2,0,0)$ is locally isomorphic to
\be
S^3\x S^3.
\ee
\end{prop}
\begin{proof}
Introducing the linearly independent $1$-forms
\ba
u_1 &:= 2\lb(a_1\cdot e_1+a_2\cdot e_2+\sqrt{a_1^2+a_2^2}\cdot e_4\rb), &
u_2 &:= 2\lb(-a_2\cdot e_1+a_1\cdot e_2+\sqrt{a_1^2+a_2^2}\cdot e_3\rb),\\
u_3 &:= A_2+2\sqrt{a_1^2+a_2^2}\cdot e_5,
\ea
\ba
v_1 &:= 2\lb(a_1\cdot e_1+a_2\cdot e_2-\sqrt{a_1^2+a_2^2}\cdot e_4\rb), &
v_2 &:= 2\lb(-a_2\cdot e_1+a_1\cdot e_2-\sqrt{a_1^2+a_2^2}\cdot e_3\rb), \\
v_3 &:= A_2-2\sqrt{a_1^2+a_2^2}\cdot e_5,
\ea
system ($\star$), together with $dA_2=\alpha\cdot F$, is equivalent to
\ba
du_1&=-u_2\wedge u_3, & du_2&=-u_3\wedge u_1, & du_3&=-u_1\wedge u_2,\\
dv_1&=-v_2\wedge v_3, & dv_2&=-v_3\wedge v_1, & dv_3&=-v_1\wedge v_2.
\ea
\end{proof}
\begin{exa}\label{exa:1}
As a result of proposition \ref{prop:7}, the $5$-dimensional Stiefel manifold (see \cite{Sti35})
\be
V_{4,2}=\lb\{(x,y)\in\R^4\x\R^4\setsep \lan x,x\ran=\lan y,y\ran=1,\,\lan x,y\ran=0\rb\}\subset S^3\x S^3,\quad \lan x,y\ran=\sum_i x_iy_i,
\ee
admits a non-integrable almost contact metric structure of class $\W_4$. The corresponding Nijenhuis tensor is non-trivial and totally skew-symmetric (cf.\ remarks \ref{rmk:1} and \ref{rmk:2}).
\end{exa}
\begin{prop}\label{prop:6}
If $a_3^2+a_4^2\neq0$, the Lie group $G(0,0,a_3,a_4)$ is locally isomorphic to
\be
\widetilde{\SL(2,\R)}\x\widetilde{\SL(2,\R)}.
\ee
\end{prop}
\begin{proof}
We define
\ba
u_1 &:= \sqrt{2}\lb(a_3\cdot e_1+a_4\cdot e_2+\sqrt{a_3^2+a_4^2}\cdot e_4\rb), &
u_2 &:= \sqrt{2}\lb(-a_4\cdot e_1+a_3\cdot e_2+\sqrt{a_3^2+a_4^2}\cdot e_3\rb),\\
u_3 &:= A_2+\sqrt{a_3^2+a_4^2}\cdot e_5
\ea
and
\ba
v_1 &:= \sqrt{2}\lb(a_3\cdot e_1+a_4\cdot e_2-\sqrt{a_3^2+a_4^2}\cdot e_4\rb), &
v_2 &:= \sqrt{2}\lb(-a_4\cdot e_1+a_3\cdot e_2-\sqrt{a_3^2+a_4^2}\cdot e_3\rb), \\
v_3 &:= A_2-\sqrt{a_3^2+a_4^2}\cdot e_5.
\ea
Using these linearly independent $1$-forms, the equations of system ($\star$), together with $dA_2=\alpha\cdot F$, transform to
\ba
du_1&=-u_2\wedge u_3, & du_2&=-u_3\wedge u_1, & du_3&=u_1\wedge u_2,\\
dv_1&=-v_2\wedge v_3, & dv_2&=-v_3\wedge v_1, & dv_3&=v_1\wedge v_2.
\ea
\end{proof}
\begin{prop}\label{prop:4}
Let $a_1,\ldots,a_4$ be constants such that $a_2=a_4=0$ and
\bite
\item[a)] $a_1=a_3\neq0$.
\item[b)] $-2a_1=a_3\neq0$.
\item[c)] $(a_1-a_3)(2a_1+a_3)>0$.
\item[d)] $(a_1-a_3)(2a_1+a_3)<0$.
\eite
Then the corresponding Lie group $G(a_1,a_2,a_3,a_4)$ is locally isomorphic to
\bite
\item[a)] $\R^6$.
\item[b)] $H^5\x\R$.
\item[c)] $S^3\x S^3$.
\item[d)] $\widetilde{\SL(2,\R)}\x\widetilde{\SL(2,\R)}$.
\eite
Here, $H^5$ denotes the $5$-dimensional Heisenberg group.
\end{prop}
\begin{proof}
Suppose $a_1=a_3\neq0$ and $a_2=a_4=0$. In this case, we have to solve
\ba
de_1&=A_2\wedge e_2-3a_1\cdot e_3\wedge e_5, & de_2&=-A_2\wedge e_1+3a_1\cdot e_4\wedge e_5, \\
de_3&=-A_2\wedge e_4+3a_1\cdot e_1\wedge e_5, & de_4&=A_2\wedge e_3-3a_1\cdot e_2\wedge e_5, \\
de_5&=0, &dA_2&=0.
\ea
Locally, each of the $1$-forms $A_2+3a_1\cdot e_5$, $A_2-3a_1\cdot e_5$ is the differential of some function,
\ba
A_2+3a_1\cdot e_5&=df, & A_2-3a_1\cdot e_5&=dg.
\ea
Introducing a new frame $(u_1,\ldots,u_6)$ by
\ba
u_1 &:= \cos(f)(e_1+e_4)-\sin(f)(e_2+e_3), &
u_2 &:= \sin(f)(e_1+e_4)+\cos(f)(e_2+e_3), \\
u_3 &:= \cos(g)(e_1-e_4)-\sin(g)(e_2-e_3), &
u_4 &:= \sin(g)(e_1-e_4)+\cos(g)(e_2-e_3), \\
u_5 &:= A_2+3a_1\cdot e_5, &
u_6 &:= A_2-3a_1\cdot e_5,
\ea
we obtain
\ba
du_1&=0, & du_2&=0, & du_3&=0, & du_4&=0, & du_5&=0, & du_6&=0.
\ea
We now suppose $-2a_1=a_3\neq0$ and $a_2=a_4=0$. Here, the equations reduce to
\ba
de_1&=A_2\wedge e_2, & de_2&=-A_2\wedge e_1, & de_3&=-A_2\wedge e_4, & de_4&=A_2\wedge e_3,
\ea
\ba
de_5&=-6a_1(e_1\wedge e_3-e_2\wedge e_4), & dA_2&=0.
\ea
Again, we introduce a new frame $(u_1,\ldots,u_6)$,
\ba
u_1 &:= \cos(f)(e_1+e_4)-\sin(f)(e_2+e_3), &
u_2 &:= \sin(f)(e_1+e_4)+\cos(f)(e_2+e_3), \\
u_3 &:= \sin(f)(e_1-e_4)+\cos(f)(e_2-e_3), &
u_4 &:= \cos(f)(e_1-e_4)-\sin(f)(e_2-e_3), \\
u_5 &:= -\frac{2}{3a_1}e_5, &
u_6 &:= A_2 = df.
\ea
Using $(u_1,\ldots,u_6)$, the equations transform to
\ba
du_1&=0, & du_2&=0, & du_3&=0, & du_4&=0, & du_5&=2(u_1\wedge u_2+u_3\wedge u_4), & du_6&=0.
\ea
Finally, we suppose $(a_1-a_3)(2a_1+a_3)\neq0$ and $a_2=a_4=0$. In this case, we define 
\ba
u_1 &:= e_1+e_4, &
u_2 &:= e_2+e_3, &
u_3 &:= A_2+(2a_1+a_3)e_5,\\
v_1 &:= e_1-e_4, &
v_2 &:= e_2-e_3, &
v_3 &:= A_2-(2a_1+a_3)e_5,
\ea
and system ($\star$), together with $dA_2=\alpha\cdot F$, is equivalent to
\ba
du_1&=-u_2\wedge u_3, & du_2&=-u_3\wedge u_1, & du_3&=-2(a_1-a_3)(2a_1+a_3)u_1\wedge u_2,\\
dv_1&=-v_2\wedge v_3, & dv_2&=-v_3\wedge v_1, & dv_3&=-2(a_1-a_3)(2a_1+a_3)v_1\wedge v_2.
\ea
\end{proof}
\begin{rmk}\label{rmk:3}
After swapping $a_1$ and $a_2$ as well as $a_3$ and $a_4$, an analogous proof shows that proposition \ref{prop:4} remains valid.
\end{rmk}
%
%
%
%

%
%
%
\end{document}